\documentclass[10pt]{amsart}

\usepackage{amsmath, amscd, amssymb}
\usepackage[frame,cmtip,arrow,matrix,line,graph,curve]{xy}
\usepackage{graphpap, color}
\usepackage[mathscr]{eucal}
\usepackage{cancel}
\usepackage{verbatim}

\numberwithin{equation}{section}

\def\and{\quad{\rm and}\quad}

\newtheorem{prop}{Proposition}[section]
\newtheorem{theo}[prop]{Theorem}

\newtheorem{rema}[prop]{Remark}

\newtheorem{conj}[prop]{Conjecture}

\def\beq{\begin{equation}}
\def\eeq{\end{equation}}

\title[Motives of moduli spaces of rank 2 vector bundles on curves]{Remarks on motives of moduli spaces \\ of rank 2 vector bundles on curves}

\author{Kyoung-Seog Lee}

\address{Center for Geometry and Physics, Institute for Basic Science (IBS), Pohang 37673, Republic of Korea}
\email{kyoungseog02@gmail.com}

\thanks{KSL was partially supported by IBS-R003-Y1.}

\begin{document}

\begin{abstract}
Let $C$ be an algebraic curve of genus $g \geq 2$ and $M_L$ be the moduli space of rank 2 stable vector bundles on $C$ whose determinants are isomorphic to a fixed line bundle $L$ of degree 1 on $C.$ In \cite{Bano2}, S. del Bano studied motives of moduli spaces of rank 2 vector bundles on $C$ and computed the motive of $M_L.$ In this note, we prove that his result gives an interesting decomposition of the motive of $M_L.$ This motivic decomposition is compatible with a conjecture of M. S. Narasimhan which predicts semi-orthogonal decomposition of derived category of the moduli space.
\end{abstract}
\maketitle

\section{Introduction}

Let $C$ be a smooth projective curve of genus $g \geq 2$ over $\mathbb{C}$ and $M_L$ be the moduli space of rank 2 stable vector bundles on $C$ whose determinants are isomorphic to a fixed line bundle $L$ of degree 1 on $C.$ Let $E$ be the Poincar\'{e} bundle on $C \times M_L.$ We denote by $D(C)$(resp. $D(M_L)$) the bounded derived category of coherent sheaves on $C$(resp. $M_L$). Then $E$ defines a functor 
$$ \Phi_E : D(C) \to D(M_L) $$
by $\Phi_E(F) := {Rp_{M_L}}_*({Lp_C}^*(F) \otimes^L E)$, where $F$ is an element in $D(C)$ and $p_C$(resp. $p_{M_L}$) is the projection map from $C \times M_L$ to $C$(resp. $M_L$). It was proved that $\Phi_E$ is fully faithful for every smooth projective curve of genus greater than or equal to 2 in \cite{BO, FK, Narasimhan1, Narasimhan2}. Moreover the results in \cite{Narasimhan1} imply that there is the following semi-orthogonal decomposition.
$$ D(M_L) = \langle \langle D(pt), D(pt), D(C) \rangle^{\perp}, \langle D(pt), D(pt), D(C) \rangle \rangle $$

It is an interesting and important task to understand the semi-orthogonal component $\langle D(pt), D(pt), D(C) \rangle^{\perp}.$ It was conjectured by M. S. Narasimhan that the derived category of $M_L$ has a semi-orthogonal decomposition consisting of two copies of the derived category of point, two copies of the derived category of $C,$ $\cdots,$ two copies of the derived category of $C^{(g-2)}$ and one copy of the derived category of $C^{(g-1)},$ where $C^{(n)}$ denotes the $n$-th symmetric power of $C.$

\begin{conj}\label{mainconjecture}
The derived category of $M_L$ has the following semi-orthogonal decomposition
$$ D(M_L)=\langle D(pt), D(pt), D(C), D(C), \cdots, D(C^{(g-2)}), D(C^{(g-2)}), D(C^{(g-1)}) \rangle. $$
\end{conj}

Using del Bano's work(cf. \cite{Bano2}) we found that the motive of $M_L$ has the following interesting decomposition. Assuming some conjectures about derived categories and motives (cf. \cite{Orlov}), the following motivic decomposition is compatible with the above conjecture.

\begin{theo}
In any semisimple category of motives, there is the following isomorphism.
$$ h(M_L) \cong \bigoplus_{k=0}^{g-2} h(C^{(k)}) \otimes (\mathbb{L}^{\otimes k} \oplus \mathbb{L}^{\otimes 3g-3-2k}) \oplus h(C^{(g-1)}) \otimes \mathbb{L}^{\otimes g-1}. $$
\end{theo}

The author was informed that P. Belmans, S. Galkin and S. Mukhopadhyay obtained the conjecture \ref{mainconjecture} and a similar stronger formula in the Grothendieck group of integral K-motives independently.

\medskip

It is a pleasure to express author's deepest gratitude to M. S. Narasimhan for sharing his conjecture and many invaluable teachings and encouragements. The author is also very grateful to T. L. Gomez for many helpful discussions and encouragements. He also thanks P. Belmans, C.-H. Cho, S. Galkin, A. Kuznetsov, C. L. Lazaroiu, S. Mukhopadhyay, S. Okawa and K.-D. Park for helpful discussions and suggestions. Part of this work was done while he was a research fellow of KIAS and visiting Indian Institute of Science. He thanks Indian Institute of Science for wonderful working conditions and kind hospitality. He also thanks G. Misra for kind hospitality during his stay in IISc.

\medskip

\noindent\textbf{Notation}. We will follow the notations in \cite{Bano2}.

\section{Motivic decomposition}

We refer \cite{Bano1, Bano2} for notations and backgrounds about motives. In \cite{Bano1}, Bano proved the following motivic decomposition.

\begin{theo}
Let $C^{(n)}$ be the $n$-th symmetric power of $C.$ Then there is the following decomposition.
$$ h(C^{(n)}) = \bigoplus_{a+b+c=n} 1^{\otimes a} \otimes \lambda^b h^1(C) \otimes \mathbb{L}^{\otimes c} $$
where $a,b,c$ are nonnegative integers.
\end{theo}

In \cite{Bano2}, Bano proved the following motivic decomposition using works of Thaddeus.

\begin{theo}\cite[Corollary 2.7]{Bano2}
In any semisimple category of motives, there is the following isomorphism.
$$ h(M_L) \cong \bigoplus_{k=0}^{g} \lambda^kh^1C \otimes (1 \oplus \mathbb{L} \oplus \cdots \oplus \mathbb{L}^{\otimes g-k-1}) \otimes (1 \oplus \mathbb{L}^{\otimes 2} \oplus \cdots \oplus \mathbb{L}^{\otimes 2g-2k-2}) \otimes \mathbb{L}^{\otimes k}. $$
\end{theo}

The main observation of this note is the following.

\begin{theo}
In any semisimple category of motives, there is the following isomorphism.
$$ h(M_L) \cong \bigoplus_{k=0}^{g-2} h(C^{(k)}) \otimes (\mathbb{L}^{\otimes k} \oplus \mathbb{L}^{\otimes 3g-3-2k}) \oplus h(C^{(g-1)}) \otimes \mathbb{L}^{\otimes g-1}. $$
\end{theo}
\begin{proof}
It is enough to compare 
$$ \bigoplus_{k=0}^{g} \lambda^kh^1C \otimes (1 \oplus \mathbb{L} \oplus \cdots \oplus \mathbb{L}^{\otimes g-k-1}) \otimes (1 \oplus \mathbb{L}^{\otimes 2} \oplus \cdots \oplus \mathbb{L}^{\otimes 2g-2k-2}) \otimes \mathbb{L}^{\otimes k} $$
and
$$ \bigoplus_{k=0}^{g-2} h(C^{(k)}) \otimes (\mathbb{L}^{\otimes k} \oplus \mathbb{L}^{\otimes 3g-3-2k}) \oplus h(C^{(g-1)}) \otimes \mathbb{L}^{\otimes g-1}. $$
Then it is enough to compare the coefficients of $\lambda^ih^1C$ for both sides. It is easy to see that the coefficient of $\lambda^ih^1C$ is 
$$ \bigoplus_{k=i}^{g-2} \bigoplus_{a+c=k-i} (1^{\otimes a} \otimes \lambda^i h^1C \otimes \mathbb{L}^{\otimes c}) \otimes (\mathbb{L}^{\otimes k} \oplus \mathbb{L}^{\otimes 3g-3-2k}) \oplus \bigoplus_{a+c=g-1-i} (1^{\otimes a} \otimes \lambda^i h^1C \otimes \mathbb{L}^{\otimes c}) \otimes \mathbb{L}^{\otimes g-1} $$
$$ = \lambda^i h^1C \otimes [ \bigoplus_{k=i}^{g-2} \bigoplus_{a+c=k-i} (1^{\otimes a} \otimes \mathbb{L}^{\otimes c}) \otimes (\mathbb{L}^{\otimes k} \oplus \mathbb{L}^{\otimes 3g-3-2k}) \oplus \bigoplus_{a+c=g-1-i} (1^{\otimes a} \otimes \mathbb{L}^{\otimes c}) \otimes \mathbb{L}^{\otimes g-1} ] $$
$$ = \lambda^i h^1C \otimes [ \bigoplus_{k=i}^{g-2} \bigoplus_{c=0}^{k-i} (\mathbb{L}^{\otimes k+c} \oplus \mathbb{L}^{\otimes 3g-3-2k+c}) \oplus \bigoplus_{c=0}^{g-1-i}  \mathbb{L}^{\otimes g-1+c} ] $$
$$ = \lambda^i h^1C \otimes [ \bigoplus_{k=i}^{g-2} \bigoplus_{c=0}^{k-i} (\mathbb{L}^{\otimes k-i+c} \oplus \mathbb{L}^{\otimes 3g-3-3i-2(k-i)+c}) \oplus \bigoplus_{c=0}^{g-1-i}  \mathbb{L}^{\otimes g-1-i+c} ] \otimes \mathbb{L}^{\otimes i} $$
$$ = \lambda^i h^1C \otimes [ \bigoplus_{k-i=0}^{g-1-i-1} \bigoplus_{c=0}^{k-i} (\mathbb{L}^{\otimes k-i+c} \oplus \mathbb{L}^{\otimes 3g-3-3i-2(k-i)+c}) \oplus \bigoplus_{c=0}^{g-1-i}  \mathbb{L}^{\otimes g-1-i+c} ] \otimes \mathbb{L}^{\otimes i} $$
$$ = \lambda^i h^1C \otimes [ \bigoplus_{j=0}^{g-1-i-1} \bigoplus_{c=0}^{j} (\mathbb{L}^{\otimes j+c} \oplus \mathbb{L}^{\otimes 3g-3-3i-2j+c}) \oplus \bigoplus_{c=0}^{g-1-i}  \mathbb{L}^{\otimes g-1-i+c} ] \otimes \mathbb{L}^{\otimes i}. $$
Therefore it remains to show that 
$$ \bigoplus_{j=0}^{g-1-i-1} \bigoplus_{c=0}^{j} (\mathbb{L}^{\otimes j+c} \oplus \mathbb{L}^{\otimes 3g-3-3i-2j+c}) \oplus \bigoplus_{c=0}^{g-1-i}  \mathbb{L}^{\otimes g-1-i+c} $$
is equal to 
$$ (1 \oplus \mathbb{L} \oplus \cdots \oplus \mathbb{L}^{\otimes g-1-i}) \otimes (1 \oplus \mathbb{L}^{\otimes 2} \oplus \cdots \oplus \mathbb{L}^{\otimes 2g-2-2i}). $$
It follows from the following equality which holds for every $m \geq 1$ where $x$ is a formal variable.
$$ \sum_{j=0}^{m-1}\sum_{c=0}^{j}(x^{j+c}+x^{3m-2j+c}) + \sum_{c=0}^{m}x^{m+c} = \frac{1-x^{m+1}}{1-x} \cdot \frac{1-x^{2m+2}}{1-x^2}. $$
Therefore we get the desired result.
 \end{proof}

\begin{rema}
The above motivic decomposition implies an interesting decomposition of Hodge diamond of $M_L.$ 
\end{rema}

It was also suggested by M. S. Narasimhan that the conjectural semi-orthogonal decomposition of derived category of $M_L$ has a meaning in the homological mirror symmetry. Assuming some conjectures, it seems to imply an interesting decomposition of the (Karoubi completed) Fukaya category of $M_L.$ We have the following vague conjecture.

\begin{conj}
The Fukaya category of $M_L$ has the following orthogonal decomposition.
$$ \overline{Fuk(M_L)}=\langle \overline{Fuk(pt)}, \overline{Fuk(pt)}, \overline{Fuk(C)}, \overline{Fuk(C)}, \cdots, \overline{Fuk(C^{(g-2)})}, \overline{Fuk(C^{(g-2)})}, \overline{Fuk(C^{(g-1)})} \rangle $$
\end{conj}

This conjecture is compatible with results and conjectures in \cite{KKOY, Munoz, Smith}.

\bibliographystyle{amsplain}

\end{document}